\documentclass[11pt,reqno]{amsart}

\usepackage{color}

\usepackage{amssymb}
\usepackage{amsmath}
\usepackage{amsthm}
\usepackage{times}
\usepackage{graphicx}

\usepackage{color}
\usepackage{amsthm}
\newtheorem{theorem}{Theorem}

\newtheorem{proposition}{Proposition}

\usepackage[curve,all]{xy}
 \xyoption{curve}
 \xyoption{color}
 \xyoption{line}
\xyoption{arc}

\theoremstyle{definition}

\newtheorem{example}{Example}

\theoremstyle{remark}
\newtheorem{remark}{Remark}

\numberwithin{equation}{section}

\newcommand{\calC}{\mathcal{C}}
\newcommand{\calD}{\mathcal{D}}

\newcommand{\calH}{\mathcal{H}}

\newcommand{\calL}{\mathcal{L}}

\newcommand{\calN}{\mathcal{N}}

\newcommand{\calO}{\mathcal{O}}

\newcommand{\calS}{\mathcal{S}}

\newcommand{\frakS}{\mathfrak{S}}
\newcommand{\frakA}{\mathfrak{A}}

\newcommand{\bbA}{\mathbb{A}}

\newcommand{\bbV}{\mathbb{V}}

\newcommand{\bbS}{\mathbb{S}}

\newcommand{\bbF}{\mathbb{F}}
\newcommand{\bbC}{\mathbb{C}}

\newcommand{\bbP}{\mathbb{P}}

\newcommand{\bbW}{\mathbb{W}}
\newcommand{\bbU}{\mathbb{U}}

\newcommand{\bfx}{\mathbf{x}}

\newcommand{\sfA}{\mathsf{A}}
\newcommand{\sfC}{\mathsf{C}}
\newcommand{\sfS}{\mathsf{S}}

\newcommand{\sfM}{\mathsf{M}}

\newcommand{\la}{\langle}
\newcommand{\ra}{\rangle}

\newcommand{\SL}{\textrm{SL}}
\newcommand{\Aut}{\textup{Aut}}

\newcommand{\GL}{\textup{GL}}

\newcommand{\PSL}{\textup{PSL}}

\newcommand{\St}{\textup{St}}

\newcommand{\da}{\dasharrow}

\renewcommand\emptyset\varnothing

\newcommand{\beq}{\begin{equation}}
\newcommand{\eeq}{\end{equation}}

\begin{document}
\title[Quartic surfaces]{Quartic surfaces with icosahedral symmetry}

\author{Igor V. Dolgachev}
\address{Department of Mathematics, University of Michigan, 525 E. University Av., Ann Arbor, Mi, 49109, USA}
\email{idolga@umich.edu}


\begin{abstract} We study smooth quartic surfaces in $\bbP^3$ which admit a group of projective automorphisms isomorphic to the icosahedron group.
\end{abstract}

\maketitle

\section{Introduction}
Let $\frakA_5$ be the icosahedron group isomorphic to the alternating group in 5 letters. Starting from Platonic solids, it  appears as an omnipresent  symmetry group in geometry.  In this article, complementing   papers \cite{Cheltsov1}, \cite{Cheltsov2}, we  discuss families of smooth quartic surfaces in $\bbP^3$ that admit the group $\frakA_5$ as its group of projective symmetries. 

It follows from loc.cit. that any smooth quartic surface $S$ with a faithful action of $\frakA_5$ belongs to one of the two pencils of invariant quartic surfaces. One of them arises from a linear irreducible $4$-dimensional representation of $\frakA_5$ and  was studied with  great details by K. Hashimoto in \cite{Hashimoto}.  It contains a double quadric and four surfaces with 5,10,10, or 15 ordinary double points. The other pencil arises from a faithful linear representation of the binary icosahedron group $2.\frakA_5$. As was shown in \cite{Cheltsov1} and \cite{Cheltsov2}, it contains two singular surfaces singular along its own rational normal cubic and two surfaces with 10 ordinary double points. In this paper we show first that one of the surfaces with 10 ordinary double points from  Hashimoto's pencil can be realized as a Cayley quartic symmetroid defined by a $\frakA_5$-invariant web $W$ of quadrics.\footnote{This result was independently obtained by S. Mukai.} We also show that the Steinerian surface of this web parametrizing singular points of singular quadrics in the web coincides with one of two smooth member of the second pencil that admits a larger group of projective symmetries isomorphic to $\frakS_5$. To see this we give an explicit equation of the second pencil. We also show that the apolar linear system of quadrics to the web $W$ contains two invariant rational normal curves that give  rise to two rational plane sextics with symmetry $\frakA_5$ discovered by R. Winger \cite{Winger1}, \cite{Winger2}.

It is my pleasure to thank   I. Cheltsov,  B. van Geemen, K. Hulek and  S. Mukai for their help in collecting the known information about the subject of this paper. I would like also to thank the referee for careful reading of the manuscript and useful remarks.

\section{Two irreducible $4$-dimensional representations}
Let $E$ be a $4$-dimensional linear space over an algebraically closed field $\Bbbk$ of characteristic $\ne 2,3,5$ and $|E| \cong \bbP^3$ be the projective space of lines in $E$. Assume that $E$ is a non-trivial projective linear representation space for $\frakA_5 \cong \PSL(2,\bbF_5)$ in $|E|$. Then it originates from either linear representation of $\frakA_5$ in $E$, or from a linear representation of its central double extension 
$2.\frakA_5 \cong \SL(2,\bbF_5)$, the \emph{binary icosahedral group}. Then $E$ is either an irreducible representation, or contains one-dimensional trivial representation, or  decomposes into the sum of two irreducible two-dimensional representations. 

\begin{proposition} Suppose there exits  a smooth quartic surface $S$ in $|E|$ which is invariant with respect to a non-trivial projective representation of $\frakA_5$. Then $E$ is an irreducible representation of $G = \frakA_5$ or $G = 2.\frakA_5$.
\end{proposition}

\begin{proof} Suppose $E$ has one-dimensional trivial summand. It is known that an element of order 3 acting on a smooth quartic surface has 6 isolated fixed points (see \cite{Mukai}). A glance at the character table of the groups shows that an element $g$ of order 3 in $G$ in its action on  the $3$-dimension summand has three different eigenvalues, one of them is equal to $1$.  This shows that $g$ has a pointwise fixed line in $|E|$ and two isolated fixed points. The line intersects the quartic at 4 points, so all the fixed points are accounted for. In particular, one of them is fixed with respect to the whole group $\frakA_5$.  However, the assumption on the characteristic implies that $G$ acts 
faithfully on the tangent space of $S$ at this point, and the group $\frakA_5$ has no non-trivial 2-dimensional linear representations. 

Now, let us consider case when $E$ decomposes into the sum of two irreducible $2$-dimensional representations. This could happen only when $G = 2.\frakA_5$. In this case $G$ has two invariant lines in $|E|$, hence the union of two invariant  sets of $\le 4$-points on $S$. The known possible orders of subgroups of $\frakA_5$ shows that this is impossible. 
\end{proof}

So, we are dealing with projective representations of the group $\frakA_5$ coming from an irreducible four-dimensional representation $\bbW_4$ of $\frakA_5$ or from an irreducible faithful $4$-dimensional linear representation $\bbU_4$ of $2A_5$.

It is known that the group $2.\frakA_5$ has two irreducible 2-dimensional representation $\bbV$ and $\bbV'$, both self-dual but one is obtained from another by the composition with an outer automorphism of $\frakA_5$. The center acts as the minus identity. Since it acts as the identity on the symmetric squares $S^2\bbV$ or $S^2\bbV'{}$, the groups $\frakA_5,2.\frakA_5$ admit two 3-dimensional representations, both of them are irreducible and self-dual, they  differ from each other by an outer automorphism.

The group  $\frakA_5$ acts in  $|S^2\bbV^\vee|$ via the Veronese map of 
$|\bbV|\to |S^2\bbV|$ of its natural action on  $|\bbV|\cong \bbP^1$.
The restriction of the representation $S^2\bbV$ of $\frakA_5$ to its subgroup $H$ of order 10 isomorphic to the  dihedral group $D_{10}$ has trivial one-dimensional summand. It defines a fixed point of $H$ in its action in the plane $|S^2\bbV|$. Its orbit of $\frakA_5$ consists of six points,  called by F. Klein, the \emph{fundamental points}. The linear system of plane cubics is invariant with respect to $\frakA_5$ and defines an irreducible  subrepresentation of $S^3(S^2\bbV)^\vee$ of dimension $4$.  The image of the plane under the map defined by the linear system of cubics is isomorphic to the \emph{Clebsch diagonal cubic surface} representing the unique isomorphism class of a nonsingular cubic surface with the automorphism group isomorphic to $\frakS_5$ \cite{CAG}, 9.5.4. The representation $\bbW_4$ is isomorphic to the restriction of the \emph{standard irreducible representation} of $\frakS_5$ realized in the hyperplane $x_1+\cdots+x_5 = 0$ in $\Bbbk^5$. 

The representation $\bbU_4$ is realized in the third symmetric power $S^3\bbV$ of $\bbV$. It is self-dual and isomorphic to $S^3\bbV'$. The projective representation of $\frakA_4$ in $|\bbU_4|$ is obtained via the Veronese map $|\bbV|\to |S^3\bbV|$  from the natural action of $\frakA_5$ on $|\bbV|$.

Using the character table for $G$ one obtains the following.

\begin{proposition} Let $S^4\bbW^\vee$ and $S^4\bbU_4$ be the fourth symmetric power of $\bbW_4$ and of $\bbU_4$, and let $()^G$ denote  the subspace of $G$-invariant elements. Then 
$$\dim (S^4\bbW^\vee)^G = \dim (S^4\bbU_4^\vee)^G = 2.$$
\end{proposition}

Thus we have two pencils of invariant quartic polynomials $(S^4\bbW_4)^{\bbA_5}$ and $(S^4\bbU_4)^{2.\bbA_5}$, so our quartic surface $S$ is a member of one of them. 

\section{The pencil $|(S^4\bbW_4^\vee)^{\bbA_5}|$}
Since $\bbW_4$ is the restriction of the standard representation of $\frakS_5$, the invariant theory says that the space $(S^4\bbW_4)^{\bbA_5}$ is isomorphic to the linear space of symmetric polynomials of degree 4 in variables $x_1,\ldots,x_5$ and the discriminant anti-symmetric polynomial. Since the degree of the latter is larger than 4, we obtain that 
$$(S^4\bbW_4)^{\bbA_5} = (S^4\bbW_4)^{\bbS_5}.$$
Thus any $\frakA_5$-invariant  pencil of quartic surfaces in $|\bbW_4|$ consists of the quartics given by  equations
\beq\label{eq1}
s_4+\lambda s_2^2 = s_1 = 0,
\eeq
where $s_i$ are elementary symmetric functions. It will be more convenient to rewrite this equation in terms of power symmetric functions:
\beq\label{eq2}
x_1^4+\cdots+x_6^4-t(x_1^2+\cdots+x_6^2)^2 = x_1+\cdots+x_6= 0.
\eeq 
where $t = (2\lambda+1)/2$.  

 Computing the partial derivatives, we find that the $\frakS_5$-orbits of singular points must be represented by points with coordinates $(1,1,-1,-1,0)$, or $(1,-1,0,0,0)$, or $(2,2,2,-3,-3)$, or $(1,1,1,1,-4)$ corresponding to the parameters $t = 1/4,1/2,7/30,13/20$, respectively. This gives the following. 

\begin{proposition} The surface $S_t$ given by equation \eqref{eq2} is nonsingular if and only if 
$t\ne \frac{1}{2}, \frac{1}{4},\frac{7}{30},\frac{13}{20}$.  If  $t = 1/4$, it has 15 singular points, if $t = 1/2,7/30$, it has 10 singular points, if $t= 13/20$, it has $5$ singular points. Each singular point is an ordinary node.
\end{proposition}

\begin{example}\label{exCR} The surface $S_{1/4}$ is the intersection of the \emph{Castelnuovo-Richmond-Igusa quartic  threefold} (see \cite{CAG}, 9.4.4) given by equations in $\bbP^5$
\beq\label{eq3}
x_1^4+\cdots+x_6^4-\frac{1}{4}(x_1^2+\cdots+x_6^2)^2 = x_1+\cdots+x_6= 0.
\eeq
and the hyperplane $x_6 = 0$. Its singular points are the intersection of the hyperplane with  the fifteen double lines of the threefold. 
\end{example}

\begin{example} If $t = 1/2$, then $\lambda = 0$ and the surface can be rewritten by the equation
$$\sum_{i=1}^5\frac{1}{x_i} = x_1+\cdots+x_5 = 0.$$
in $\bbP^4$. We recognize the equation of the Hessian surface of the Clebsch diagonal cubic surface $\sfC_3$. Its  10 nodes are the vertices of the Sylvester pentahedron. Its edges lie on the surface and together with nodes form the Desargues  symmetric configuration $(10_3)$ (see \cite{Hilbert}, III, \S 19). By definition of the Hessian surface, $S_{1/2}$ is a \emph{quartic symmetroid} (see \cite{CAG}, 4.2.6), the locus of points in $|E|$ such that the polar quadric of the cubic surface $\sfC_3$ is singular. This symmetric determinantal representation of $S_{1/2}$ is defined by a linear map of $\frakS_5$-representations $\bbW_4\to S^2\bbW_4^\vee$. It is  the polar map associated with the cubic surface (see \cite{CAG}). It allows one to rewrite the equation of the Hessian of the Clebsch cubic surface as the symmetric determinant:
$$\det\begin{pmatrix}L-x_1&L&L&L\\
L&L-x_2&L&L\\
L&L&L-x_3&L\\
L&L&L&L-x_4\end{pmatrix} = 0,$$
where $L = x_1+x_2+x_3+x_4$.

It is known that the surface $S_{1/2}$ is isomorphic to the K3-cover of an Enriques surface $X$ with $\Aut(X)\cong \frakS_5$ (of type VI in Kondo's classification, \cite{Kondo}). The covering involution is defined by the standard Cremona transformation $(x_1,\ldots,x_5)\mapsto (x_1^{-1},\ldots,x_5^{-1})$.
\end{example}

\begin{example}\label{seven} Assume $t = 7/30$. Projecting from the node $q_0 = (2,2,2,-3,-3)$, we get the equation of $S_{7/30}$ as the double plane
$$w^2 = (xy+xz+yz)^2(x+y+z)^2-3(x-y)^2(x-z)^2(y-z)^2 $$
$$=111x^2y^2z^2+80(x^3y^2z+x^3yz^2+x^2y^3z+x^2yz^3+xy^3z^2+xy^2z^3)$$
$$+13(x^4y^2+x^4z^2+y^4z^2+x^2z^4+y^2z^4+x^2y^4)+10(x^4yz+x^3y^3+x^3z^3+xy^4z+xyz^4+y^3z^3).$$
The branch curve $B$ is the union of two cubic curves intersecting at 9  points
$$[1,0,0],\ [0,1,0],\ [0,0,1],\ [2,-1,2],\ [-1,2,2],\ [2,2,-1],\ [-2,1,1],\ [1,-2,1],
\ [1,1,-2].$$
It  is well-known and goes back to A. Cayley (see a modern exposition in \cite{Cossec}) that this implies that the surface $S_{7/30}$ is a quartic symmetroid. We will return to this example in the next section.
\end{example}

\begin{example} Let $\calS_3$ be the \emph{Segre cubic primal} (see \cite{CAG}, 9.4.4) given by equations
\beq\label{eq1}
x_1^3+\cdots+x_6^3 = x_1+\cdots+x_6= 0
\eeq 
(see \cite{CAG}). It is isomorphic to the image of a rational map $f:\bbP^3\da \bbP^4$ given by 
the linear system of quadrics through 5 points $p_1,\ldots,p_5$ in general linear position. The 
images of the lines $\la p_i,p_j\ra$ are the ten nodes on $\calS_3$. Consider the surface $S$
in $\calS_3$ given by the additional equation 
$$x_1^2+\cdots+x_6^2 = 0.$$
Obviously, it has $\frakS_6$-symmetry. The pre-image of $S$ under the map $f$ is a 
quartic surface in $\bbP^3$ with 5 nodes at $p_1,\ldots,p_5$. One can make the map $f$ to be 
$\frakS_5$-invariant by viewing $\bbP^3$ as the hyperplane $y_0+\ldots+y_5 = 0$ in $\bbP^4$ and 
choosing the points $p_1,\ldots,p_5$ to be the points in the hyperplane with coordinates 
$[1,0,0,0,-1],[0,1,0,0,-1],$ etc. The group $\frakS_5$, acts naturally in $\bbP^3$ by permuting 
the five points. The restriction of this representation to $\frakA_5$ is the projectivization of the linear representation space $\bbW_4$. Via this action, the map $f$ becomes a $\frakS_5$-invariant birational map from $\bbP^3$ to the Segre cubic $\calC_3$. Thus $f^{-1}(S)$ is a 5-nodal quartic in $|\bbW_4|$ 
with $\frakS_5$-symmetry. 

Note that the action of $\frakS_5$ on the Segre cubic primal is closely related to its known action on the 
\emph{del Pezzo surface of degree $5$} via the following commutative diagram explained in 
\cite{Prokhorov}, Proposition 4.7:

$$\xymatrix{X\ar[dd]^{\sigma}\ar[rr]^{\chi}\ar[dr]^{\phi_0}&&X^+\ar[dd]^\phi\ar[dl]_{\phi_0^+}\\
&\calS_3&\\
\bbP^3\ar@{-->}[rr]\ar@{-->}[ur]^f&&\calD}
$$
Here $\chi$ is a flop, $\phi$ and $\phi_0^+$ are small contractions to the Segre cubic primal,  $\calD$ is a del Pezzo surface of degree $5$ and $\phi$ is a $\bbP^1$-bundle.
\end{example}

The pencil \eqref{eq2} was intensively studied by K. Hashimoto in \cite{Hashimoto}. So, it is appropriate to refer to it as the \emph{Hashimoto pencil}. In particular, Hashimoto computed the lattice of 
transcendental cycles $T(X_t)$ of a  minimal nonsingular model $X_t$ of a singular member $S_t$ of the pencil. 

\begin{theorem} Assume $\Bbbk = \bbC$. For any singular member of the Hashimoto pencil, the lattice $T(X_t)$ is of rank 2 and is  
given by the matrix
$$\begin{pmatrix}4&0\\
0&10\end{pmatrix}\  (t = \frac{1}{2}),\quad  \begin{pmatrix}4&1\\
1&4\end{pmatrix}\  (t= \frac{1}{4}),\quad \begin{pmatrix}4&2\\
2&16\end{pmatrix}\  (t = \frac{7}{30}), \quad \begin{pmatrix}6&0\\
0&20\end{pmatrix} (t = \frac{13}{20}). $$
The transcendental lattice of a generic member of the pencil is of rank 3 and is given by the matrix
$$\begin{pmatrix}4&1&0\\
1&4&0\\
0&0&-20\end{pmatrix}.$$
\end{theorem}

\begin{remark} It is known that the quartic surface defined by equation
\beq\label{maschke}
F = x^4+y^4+z^4+w^4+12xyzw = 0
\eeq
admits a group of projective automorphisms isomorphic to $2^4.\frakS_5$ (see \cite{Maschke}). According to S. Mukai \cite{Mukai}, the subgroup $2^4.\frakA_5$ is isomorphic to the Mathieu group $M_{20}$ and is realized as one of maximal finite groups of symplectic automorphisms of a complex K3 surface. The equation defining the surface is an invariant of the Heisenberg group $\calH_2$ acting in its Schr\"odinger $4$-dimensional irreducible representation (see \cite{vanGeemen}). The linear space of invariant quartic polynomials is 5-dimensional and it has a basis that consists of quartic polynomials
$$p_0= x^4+x^4+z^4+w^4,\ p_1 = x^2w^2+y^2z^2,\ p_2 = x^2z^2+y^2w^2,\ p_3 = x^2y^2+z^2w^2, \ p_4 = xyzw.$$
Let 
\begin{eqnarray*}
\sfM_1 &=&-2p_0+24p_4,\\
\sfM_2&=&p_0-6(p_1+p_2+p_3),\\
\sfM_3&=&p_0+6(-p_1+p_2+p_3),\\
\sfM_4&=&p_0+6(p_1-p_2+p_3),\\
\sfM_5&=&-2p_0-24p_4,\\
\sfM_6&=&p_0+6(p_1+p_2-p_3).
\end{eqnarray*}
be a spanning set of  this linear space. The fifth of them defines the surface from \eqref{maschke}. We propose to call  the polynomials $\sfM_i$ the \emph{Maschke quartic polynomials} (not to be confused with the Maschke octic polynomial from \cite{Bini}). It is shown in \cite{Maschke}, p. 505,  that they satisfy 
$$(\sum_{i=1}^6\sfM_i^2)^2 -4 \sum_{i=1}^6\sfM_i^4 = \sum_{i=1}^6\sfM_i = 0$$
We recognize the equations  of the Castelnuovo-Richmond-Igusa quartic threefold from Example \ref{exCR}. Thus, for any Maschke polynomial $\sfM_i$, the surface $V(\sfM_i)$ is a Galois $2^4$-cover of the pre-image of  a coordinate hyperplane section of the quartic threefold isomorphic to the surface $S_{1/4}$ from the Hashimoto pencil. This shows the appearance of $2^4.\frakS_5$ in the group of projective automorphisms of $V(\sfM_i)$.

It was communicated to me  by Bert van Geemen that the projective transformations defined by the matrices
\beq
M = \begin{pmatrix}1&1&-1&-1\\ 
1&-1&-1&1\\
-1&1&-1&1\\
-1&-1&-1&-1\end{pmatrix}, \quad N = \begin{pmatrix}1&-1&-i&-i\\ 
-1&1&-i&-i\\
i&i&1&-1\\
i&i&-1&1\end{pmatrix},
\eeq
define automorphisms of orders $5$ and $2$ that generate a subgroup of automorphism of $V(F)$ that splits the extension $2^4.\frakS_5$. This shows that the Maschke quartic surface $V(\sfM_i)$ admits $\frakS_5$ and hence $\frakA_5$ as its group of projective automorphisms. By computing the traces of $M$ and $N$ we find that $\frakS_5$ originates from its linear standard irreducible $4$-dimensional representation, and the surface must be isomorphic to a member of the  Hashimoto pencil. According to computations made by Bert van Geemen, the surface $V(\sfM_1)$ corresponds to the parameter $t = \frac{3}{20}(3-i)$.  Note that according to S. Mukai the transcendental lattice  of the Maschke surface is given by the diagonal $2\times 2$-matrix with the diagonal entries $4$ and $40$. So it is different from the transcendental lattice of a general member of the pencil.
\end{remark}

\section{The pencil $|(S^4\bbU_4^\vee)^{\frakA_5}|$}
Recall that the linear representation space  $\bbU_4$ of $G = 2.\frakA_5$ is isomorphic to the space $S^3\bbV$.  Since there is only one isomorphism class of irreducible faithful $4$-dimensional representations of $G$, we have an isomorphism $\bbU_4 \cong S^3\bbV \cong S^3\bbV'$. 

Let $e_1,e_2$ be a basis in $\bbV$ and $(u,v)$ be the coordinates in $\bbV$ with respect to this basis, i.e. the dual basis of $(e_1,e_2)$ in $\bbV^\vee$. The space $S^d\bbV$ (resp. $S^d\bbV^\vee$) has a natural monomial basis
$(e_1^d,e_1^{d-1}e_2,\ldots,e_2^d)$ (resp. $(u^d,u^{d-1}v,\ldots,v^d)$).  The polarization isomorphism 
$$S^d\bbV^\vee \to (S^d\bbV)^\vee$$
assigns to $u^{d-i}v^i$ the linear function on $S^d\bbV$ that takes the value $\frac{1}{d!}(d-i)!i!$ on $e_1^{d-i}d_2^i$ and zero on all other monomials. This shows that the basis $(\tbinom{d}{i}u^{d-i}v^i)_{i=0,\ldots,d}$ is the dual basis of $(e_1^d,e_1^{d-1}e_2,\ldots,e_2^d)$. Thus any \emph{binary form} $f\in S^d\bbV^\vee$ can be written
as 
\beq\label{binary}
f = \sum_{i=0}^d\tbinom{d}{i}a_iu^{d-i}v^i,
\eeq
so that $(a_0,\ldots,a_d)$ are the natural coordinates in the space $S^d\bbU^\vee$. Although, this notation is widely used in the invariant theory, we will switch to the basis $(\binom{d}{i}e_1^{d-i}e_2^i)$ of $S^d\bbV$ to get the dual basis formed by the monomials $u^{d-i}v^i$. So, we  will drop the binomial coefficients in \eqref{binary}. This will help us to agree our formulas with ones given in Klein's book. 

Let us clarify the coordinate-free definition of the Veronese map 
$$\nu_d:\bbV\to S^d\bbV.$$
It is defined by assigning to a vector $\alpha e_1+\beta e_2$ the linear function $f\mapsto f(\alpha,\beta)$ on 
$S^d\bbV^\vee = (S^d\bbV)^\vee$.  It follows that 
$$\nu_d(\alpha e_1+\beta e_2) = \sum_{i=0}^d\tbinom{d}{i}\alpha^{d-i}\beta^ie_1^{d-i}e_2^i = (\alpha e_1+\beta e_2)^d.$$
In coordinates, this is the map
\beq\label{par1}
(u,v)\mapsto (u^d,u^{d-1}v,\ldots,uv^{d-1},v^d).
\eeq
Passing to the projective space, we get the map
$$\nu_d:|\bbV|\to |S^d\bbV|$$
that is given by the complete linear system $|S^d\bbV^\vee| = |\calO_{|\bbV|}(d)|$. The image $R_d$ of this map is a \emph{Veronese curve} of degree $d$, or a \emph{rational normal curve} of degree $d$. If we re-denote the 
coordinates $u^{d-i}v^i$ by $(x_0,\ldots,x_d)$, a hyperplane $V(\sum_{i=0}^d a_ix_i)$ intersects $R_d$ along the closed subscheme isomorphic, under the Veronese map $\nu_d$, to the closed subscheme 
$V(\sum_{i=0}^d a_iu^{d-i}v^i)$ of $|\bbV|$.

Dually, we have the Veronese map
$$\nu_d^*:\bbV^\vee \to S^d\bbV^\vee$$
which assigns to a linear function $l=au+bv\in \bbV^\vee$ the linear function $\nu_d(l)\in S^d\bbV^\vee$ that 
takes value on $e_1^{d-i}e_2^i$ equal to $a^{d-i}b^i$. It follows that 
$$\nu_d^*(\alpha u+\beta v) = \sum_{i=0}^d\tbinom{d}{i}\alpha^{d-i}\beta^iu^{d-i}v^i = (\alpha u+\beta v)^d.$$
So, we get a familiar picture: points of $|S^d\bbV^\vee|$ are non-zero binary forms of degree $d$ up to proportionality, and points of the Veronese curve are  powers of linear forms, up to proportionality.

In coordinates, the dual Veronese map is given by
$$(e_1,e_2) \mapsto (e_1^{d},\ldots,\tbinom{d}{i}e_1^{d-i}e_2^{i},\ldots,e_2^d).$$
The image $R_d^*$ of the corresponding map 
$$\nu_d^*:|\bbV^\vee| = \bbP(\bbV)\to |S^d\bbV^\vee| = \bbP(S^d\bbV)$$
is the \emph{dual Veronese curve} of degree $d$. The duality can be clarified more explicitly. For any point 
$x\in R_d$, one can consider the \emph{osculating hyperplane} at $x$, the unique hyperplane in $|S^d\bbV|$ that intersects $R_d$ at one point $x$ with multiplicity $d$. The dual Veronese curve $R_d^*$ in the dual space 
$|S^d\bbV^\vee|$ is the locus of osculating hyperplanes. 

One can use the isomorphism $|\bbV^\vee|\to |\bbV|$ defined by assigning to a linear function $l = au+bv\in \bbV^\vee$ its zero $V(l) = [-b,a]\in |\bbV|$. In other terms, it is defined by the exterior product pairing $\bbV\times \bbV\to \bigwedge^2\bbV \cong \Bbbk$. Thus we have two Veronese maps
$$\nu_d:|\bbV| \to |S^d\bbV|, \quad \nu_d^*:|\bbV| \to |S^d\bbV^\vee|$$
with the images $R_d$ and $R_d^*$. The map of $|\bbV|\to |S^d\bbV^\vee|$ is given, in coordinates, by 
$$(u,v) \to ((-1)^de_2^d,\ldots, (-1)^{d-i}\tbinom{d}{i}e_1^ie_2^{d-i},\ldots,e_1^d).$$

Let $\rho:G\to \GL(\bbV)$ be a linear representation of a group $G$. By functoriality, it defines a linear representation $S^d(\rho):G\to \GL(S^d\bbV)$. The dual linear representation $\rho^*:G\to \GL(\bbV^\vee)$ defines a representation $S^d(\rho^*):G\to \GL(S^d\bbV^\vee)$. It follows from the polarization isomorphism that 
 the representations $S^d(\rho)$ and $S^d(\rho^*)$ are dual to each other.
 
After we fix these generalities, it is easy to describe irreducible linear representations of $G = 2.\frakA_5$.
We start with the $2$-dimensional representations $\bbV$ and $\bbV^\vee$. We choose a basis $e_1,e_2$ in $\bbV$ and its dual basis $(u,v)$ in $\bbV^\vee$ as above to assume that the group preserves the volume forms $e_1\wedge e_2$ and $u\wedge v$. Thus in these bases we represent the matrices of the representation by unimodular matrices.  
According to \cite{Klein}, p. 213, the group $2.\frakA_5$ is generated by the transformations $S,T,U$ of orders 
$5,4,4$, respectively, and its representation in $\bbV$ is given in terms of coordinates as follows:  
\begin{eqnarray*}
S:(u,v)&\mapsto& (\epsilon^3u,\epsilon^2v),\\
T:(u,v)&\mapsto&\frac{1}{\sqrt{5}}(-cu+dv,du+cv),\\
U:(u,v)&\mapsto&(-v,u),
\end{eqnarray*}
where $\epsilon = e^{2\pi i/5}$ and 
$$c = \epsilon-\epsilon^{-1},\ d = \epsilon^2-\epsilon^{-2}.$$
We have 
$$\lambda = c/d = \epsilon+\epsilon^{-1}+1 = \frac{1+\sqrt{5}}{2}$$
is the golden ratio. It  satisfies 
$\lambda^2 = \lambda+1.$
Note that the trace of $S$ is equal to $\epsilon^3+\epsilon^2 = \lambda^2-2 = \lambda-1 = \frac{-1+\sqrt{5}}{2}$. This is denoted by $b5$ in \cite{ATLAS}, so we can identify this representation with the one given by the character $\chi_6$.  The representation $\bbV'$  is given by the same formulas as above, where $\epsilon$ is replaced with $\epsilon^2$. The trace of $S$ becomes $b5^*:=-\lambda$.

Next we consider the $3$-dimensional irreducible representations realized in   $S^2\bbV$ and $S^3\bbV'$. In the basis $(e_1^2,2e_1e_2,e_2^2)$ of $\bbV$, the first representation is given by the matrices
$$
S:\begin{pmatrix}\epsilon&0&0\\
0&1&0\\
0&0&\epsilon^{-1}\end{pmatrix},\ 
T:\frac{1}{\sqrt{5}}\begin{pmatrix}\epsilon+\epsilon^{4}&2&\epsilon^2+\epsilon^{3}\\
1&1&1\\
\epsilon^2+\epsilon^{3}&2&\epsilon+\epsilon^{4}\end{pmatrix},\ 
U:\begin{pmatrix}0&0&1\\
0&-1&0\\
1&0&0\end{pmatrix}.
$$
Note that Klein uses  slightly different coordinates $(\sfA_0,\sfA_1,\sfA_2) = (-uv,v^2,-u^2)$ so his matrices are slightly different.

The second irreducible $3$-dimensional representation $(S^2\bbV)'$ is obtained by replacing $\epsilon$ with $\epsilon^2$. 
The invariant theory for the icosahedron group $\frakA_5$ in this representation is discussed in Klein's book (see also \cite{CAG}, 9.5.4). 

The linear representation $\bbU_4$ of $2.\frakA_5$ realized in $S^3\bbV$ is given by the formulas
\begin{eqnarray*} 
S&:& \biggl(\begin{smallmatrix}\epsilon^4&0&0&0\\
0&\epsilon^3&0&0\\
0&0&\epsilon^2&0\\
0&0&0&\epsilon\end{smallmatrix}\biggr),\\
T&:&\frac{d^3}{5\sqrt{5}}\biggl(\begin{smallmatrix}-\lambda^3&3\lambda^2&-3\lambda&1\\
\lambda^2&-2\lambda+\lambda^3&1-2\lambda^2&\lambda\\
-\lambda&1-2\lambda^2&2\lambda-\lambda^3&\lambda^2\\
1&3\lambda&3\lambda^2&\lambda^3\end{smallmatrix}\biggr) = 
\frac{d^3}{5\sqrt{5}}\biggl(\begin{smallmatrix}-(1+2\lambda)&3(\lambda+1)&-3\lambda&1\\
\lambda+1&1&-2\lambda-1&\lambda\\
-\lambda&-2\lambda-1&-1&\lambda+1\\
1&3\lambda&3(\lambda+1)&2\lambda+1\end{smallmatrix}\biggr),\\
U &= &\biggl(\begin{smallmatrix}0&0&0&-1\\
0&0&1&0\\
0&-1&0&0\\
1&0&0&0\end{smallmatrix}\biggr).
\end{eqnarray*}
The dual representation is given by the same formulas, where $\epsilon$ is replaced with 
$\epsilon^{2}$.

 We will use the coordinates 
$$(x_0 = u^3,x_1 = u^2v,x_2 = uv^2,x_3 = v^3).$$
Let $\calN_1\subset S^2\bbU_4^\vee$ be the linear space of quadratic forms such that $|\calN_1|$ is the linear system of quadrics with the base locus equal to the Veronese curve $R_3$. Obviously, it is an irreducible  summand of $S^2\bbU_4^\vee$. It is generated by the quadratic forms 
$$q_1 = x_0x_3-x_1x_2, \quad  q_2 = x_0x_2-x_1^2, \quad  q_3 = x_1x_3-x_2^2.$$
Under the transformation $S$ they are multiplied by $1,\epsilon, \epsilon^4$, respectively. The trace of $S$ is equal to $1+\epsilon+\epsilon^4 = \lambda$. Thus, we can identify the space $\calN_1$ with the linear representation $S^2\bbV$. Let $\calN_2^*$ be the linear space of quadratic forms in the dual space $\bbU_4^\vee$ such that the linear system of quadrics $|\calN_2^*|$ has the base locus equal to the dual Veronese curve $R_3^*$. It is spanned by the quadratic forms 
$$q_1' = 9\xi_0\xi_3-\xi_1\xi_2,\quad q_2' = 3\xi_0\xi_2-\xi_1^2, \quad q_3' =  3\xi_1\xi_3-\xi_2^2,$$
where $\xi_0 = e_1^3,\xi_1 = e_1^2e_2,\xi_2 = e_1e_2^2,\xi_3 = e_2^3$ are the dual coordinates. The representation $\calN_2^*$ is isomorphic to $S^2\bbV'$. Consider the dual space $(\calN_2^*)^\perp \subset S^2\bbU_4^\vee$ of apolar quadratic forms. It is spanned by quadratic forms
$$p_1 = x_0x_3+9x_1x_2, \ p_2 = 2x_0x_2+3x_1^2, \ p_3 = 2x_1x_3+3x_2^2, \ p_4 =x_0^2, \ p_5 = x_3^2, \ p_6 = x_0x_1, \ p_7 = x_2x_3.$$
The linear span is a 7-dimensional summand of $S^2\bbU_4$. Computing the character of $S^2\bbU^\vee$, we find the        decomposition
\beq\label{decomp1}
S^2\bbU_4^\vee \cong S^2\bbU_4 \cong \bbW_4\oplus S^2\bbV\oplus S^2\bbV'.
\eeq
This shows that  
$$(\calN_2^*)^\perp = \bbW_4\oplus S^2\bbV'.$$
Observe that  the quadratic forms $p_i$ are eigenvectors of $S$ with eigenvalues 
$1,\epsilon, \epsilon^4, \epsilon^3, \epsilon^2, \epsilon^2, \epsilon^3 $, respectively. Since $S$ has trace $-1$ on $\bbW_4$, we find that the summand $V_1$ of $(\calN_2^*)^\perp$ isomorphic to $\bbW_4$ is spanned by eigenvectors with eigenvalues $\epsilon, \epsilon^2,\epsilon^3, \epsilon^4$.  
 Under the transformation $T$, the forms $(p_1,\ldots,p_7)$ are transformed to $(-p_1,p_3,p_2,p_5,p_4,-p_7,-p_6).$  
Now we use the following known fact due to  T. Reye (see \cite{Hulek}, Lemma 4.3).

\begin{theorem}[T. Reye] A general $6$-dimensional linear space $L$ of quadrics in $\bbP^3 = |E|$ contains precisely two nets $\calN_1,\calN_2$ with the base loci equal  to  Veronese curves $C_1,C_2$ of degree 3. The dual space $L^\perp$ of apolar quadrics in $|E^\vee|$ contains 10 quadrics $Q_i$ with one-dimensional singular locus $\ell_i$. Each line $\ell_i$ is a common secant of $C_1,C_2$.
\end{theorem}

It follows from the lemma that the summand of $S^2\bbU_4^\vee$ isomorphic to $S^2\bbV'$ is isomorphic to the linear space $\calN_2$ such that the base locus of the net of quadrics $|\calN_2|$ is  a rational normal cubic curve. It is easy to see that this is possible only if $\calN_2$ is generated by 
$$r_1 = x_0x_3+9x_1x_2,\  r_2 = x_0^2+ax_2x_3, \ r_3 = x_3^2-ax_0x_1,$$
where $a^2 = 9$. Since the image of the point $[0,1,0,0]$ on the curve under the transformation $T$ must be on the curve, we check that $a = 3$. Thus, the base locus is a rational normal cubic curve with parametric equations
\beq\label{par2}
|\bbU| \to |\bbU_4|, \ (u,v) \mapsto (9u^2v,27v^3,-u^3,27uv^2).
\eeq
Note that one can avoid using Reye's Theorem by using decomposition \eqref{decomp1} and the fact there exists no $\frakA_5$-invariant lines, $\frakA_5$-invariant conics, and $\frakA_5$-orbits are of length $\le 8$.

Now we are ready to write the pencil $|S^4\bbU_4^\vee|$ of invariant quartics. We take as generators of the pencil  the tangential ruled surfaces of the invariant rational normal curves defined by the nets $|\calN_1|$ and $|\calN_2|$. They are obviously invariant with respect to the action of $G$ in $|\bbW_4|$. The net $|\calN_1|$ defines a map from $|\bbW_4|$ to $|\calN_1^\vee| = |S^2\bbV^\vee|$ whose image is an invariant conic, the \emph{fundamental conic}.  Its equation in Klein's coordinates $\sfA_0,\sfA_1,\sfA_2$ is $\sfA_0^2+\sfA_1\sfA_2 = 0$. The equation of the dual conic in the dual coordinates $\sfA_0',\sfA_1',\sfA_2'$ in  the dual plane $|S^2\bbU^\vee|$ is 
$$\sfA_0'{}^2+4\sfA_1'\sfA_2' = 0.$$ 
It follows that the equation of the tangential ruled surface must be of the form $Q_0^2+4Q_1Q_2$, where $Q_i = 0$ are equations of quadrics from the net $|\calN_1|$. We know that $A_0,A_1,A_2$ are eigenvectors for the transformation $S$ acting in $|S^2\bbU|$ with eigenvalues $1,\epsilon,\epsilon^{-1}$. Thus $A_0',A_1',A_2'$ are eigenvectors for the action of $S$ in $S^2\bbU^\vee$ with the eigenvalues $1,\epsilon^{-1},\epsilon$. Therefore our quadrics $Q_0,Q_1,Q_2$ are also eigenvectors for the action of $S$ in $S^2\bbW_4^\vee$ with eigenvalues 
$1,\epsilon^{-1},\epsilon$. We find that  
$(Q_0,Q_1,Q_2) = (q_1,\lambda q_3,\lambda'q_2)$
for some constants $\lambda,\lambda'$. Since the equation is invariant with respect to the transformation $U$, and $\sfA_1'\mapsto -\sfA_2'$, we must have $\lambda = -\lambda'$. We find the condition on $\lambda$ that guarantees that the equation $Q_0^2-\lambda^2Q_1Q_2 = 0$ is the equation of the tangential ruled surface of the Veronese curve $C_1 = R_3$.

At each point $[1,t_0,t_0^2,t_0^3]$ the tangent line to the Veronese curve $R_3$ is spanned by the vector $[0,1,2t_0,3t_0^2]$.  Plugging in the parametric equation
$s(1,t_0,t_0^2,t_0^3)+r(0,1,2t_0,3t_0^2)$ of the tangent line in the equation $q_1^2+c^2q_2q_3 = 0$, we obtain
that $c^2-4 = 0$. Thus the equation of the quartic tangential ruled surface is 
\begin{eqnarray}\label{tang1}
\sfS_1&:&(x_0x_3-x_1x_2)^2-4(x_0x_2-x_1^2)(x_1x_3-x_2^2)\\ \notag
&=& x_0^2x_3^2-6x_0x_1x_2x_3+4x_0x_2^3+4x_1x_3^3-3x_1^2x_2^2 = 0. \notag
\end{eqnarray}
The second invariant quartic is the tangential ruled surface of the rational normal curve $C_2$ defined parametrically in \eqref{par2}. The net of quadrics containing this curve is $G$-equivariantly isomorphic to $S^2\bbV'$. It defines the map from $|\bbU_4|$ to $|(S^2\bbV')^\vee|$. The invariant conic in this space is the same conic as in the previous case. Similarly to this case we get the equation
\beq\label{tang}
(x_0x_3+9x_1x_2)^2+d^2(x_0^2+3x_2x_3)(x_3^2-3x_0x_1) = 0
\eeq 
The parametric equation of the tangent line to the base curve $C_2$ is 
$$s(9t_0^2,27,-t_0^3,27t_0)+r(6t_0,0,-t_0^2,9) = 0.$$
Plugging this in the equation \eqref{tang}, we obtain that $d^2+4 = 0$ and the equation of the second tangential surface is 
\begin{eqnarray}\label{tang2}
\sfS_2&:&(x_0x_3+9x_1x_2)^2-4(x_0^2+3x_2x_3)(x_3^2-3x_0x_1)\\ \notag
&=&3(4x_0^3x_1-x_0^2x_3^2+18x_0x_1x_2x_3+27x_1^2x_2^2-4x_2x_3^3) =0. \notag
\end{eqnarray}
So, our pencil $|(S^4\bbU_4^\vee)^{2.\frakA_5}|$ can be explicitly written in the form 
$$\lambda F_1+\mu F_2 = 0,$$
where 
 $$F_1 = x_0^2x_3^2-6x_0x_1x_2x_3+4x_0x_2^3+4x_1x_3^3-3x_1^2x_2^2$$
 and 
 $$F_2 =  
4x_0^3x_1-x_0^2x_3^2+18x_0x_1x_2x_3+27x_1^2x_2^2-4x_2x_3^3.$$

Consider the linear transformation
\beq\label{mapK}
K:(x_0,x_1,x_2,x_3)\mapsto (\sqrt{3}x_2,\frac{1}{\sqrt{3}}x_0,-\frac{1}{\sqrt{3}}x_3,\sqrt{3}x_1).
\eeq
We have $K^2 = U\in G$, and the group generated by $K$ and $G$ is isomorphic to $2.\frakS_5$. We immediately check that it transforms
$$x_0x_3+9x_1x_2\mapsto 3(-x_0x_3+x_1x_2), \ x_0^2+3x_2x_3\mapsto 3(x_1^2-x_0x_2), \ x_3^2-3x_0x_1\mapsto 3(x_2^2-x_1x_3).$$
This shows that $K(\sfS_2) = \sfS_1$, more precisely, we get $K(F_2) = 3F_1$. 
Thus  the surfaces 
\begin{eqnarray}\label{s3s4}
\sfS_3 &=& V(3F_1+F_2) = V(x_0^2x_3^2+2x_0^3x_1+9x_1^2x_2^2-2x_2x_3^3+6x_0x_2^3+6x_1^3x_3),\\ \notag
\sfS_4 &=& V(F_2-3F_1) = V(-x_0^2x_3^2+9x_0x_1x_2x_3-3x_0x_2^3-3x_1^3x_3+9x_1^2x_2^2+x_0^3x_1-x_2x_3^3). \notag
\end{eqnarray}
are $\frakS_5$-invariant. One checks, using MAPLE, that they  are nonsingular.

The following result is proven in \cite{Cheltsov2}.

\begin{proposition}\label{base} The pencil generated by the quartic surfaces $\sfS_1$ and $\sfS_2$ contains two more singular members, each of them has 10 ordinary double points. The base curve $B$ of the pencil is an irreducible  curve of 
degree 16 with 24 singular points, each of them is an ordinary cusp. It also contains a unique $\frakA_5$-orbit of 20 points.
\end{proposition}

It is easy to see these two orbits of singular points of $B$. Each of them is the intersection of one of the rational normal cubic curves $C_1$ and $C_2$ with a nonsingular surface from the pencil. If we take the latter to be the surface $\sfS_3$ or $\sfS_4$, we obtain that the union of the two orbits is invariant with respect to $\frakS_5$.

 In the next section we will be able to see explicitly the unique orbit of 20 points.
 
\section{An $\frakA_5$-invariant web of quadrics in $|\bbU_4|$}
Consider the linear $2.\frakA_5$-equivariant  map 
\beq\label{web}
\phi:\bbW_4\to S^2\bbU_4^\vee
\eeq
defined by the decomposition \eqref{decomp1}. The corresponding map of the projective spaces defines a web of quadrics in $|\bbU_4|$. Recall that some of its attributes are the determinantal surface $D(\phi)$ parametrizing singular quadrics in the web, and the Steinerian surface $\St(\phi)$ parametrizing singular points of singular quadrics from the web. 

Let us find the equation of the discriminant surface. We choose a basis of $\phi(\bbW^4)$ formed by the 
quadrics
$$Q_1 = 2x_1x_3+3x_2^2,\ Q_2 = x_0^2-2x_2x_3,\ Q_3 = x_3^2+2x_0x_1,\ Q_4 = 3x_1^2+2x_0x_2.$$
They are transformed under the representation  of $2.\frakA_5$ (with the center acting trivially) as follows
\begin{eqnarray*}
S&:&(Q_1,Q_2,Q_3,Q_4)\mapsto (\epsilon^4Q_1,\epsilon^3Q_2,\epsilon^2Q_3,\epsilon Q_4),\\
T&:&\begin{pmatrix}Q_1\\
Q_2\\
Q_3\\
Q_4\end{pmatrix} \mapsto c\begin{pmatrix}-\lambda&1&\lambda^2&\lambda\\
1&\lambda&-\lambda&\lambda^2\\
\lambda^2&-\lambda&\lambda&1\\
\lambda&\lambda^2&1&-\lambda\end{pmatrix} \begin{pmatrix}Q_1\\
Q_2\\
Q_3\\
Q_4\end{pmatrix},\\
U&:&(Q_1,Q_2,Q_3,Q_4)\mapsto (Q_4,Q_3,Q_2,Q_1),
\end{eqnarray*}
where $c$ is some constant which  will not be of concern for us.

Choose a basis of $\phi(\bbW_4) \subset S^2\bbU_4^\vee$ spanned by the following quadratic forms:
\begin{eqnarray}\label{basis1}
Q_1' &= &\epsilon^4Q_1-\epsilon^3Q_2-\epsilon^2Q_3+\epsilon Q_4,\\ \notag
Q_2' &=& -\epsilon^3Q_1-\epsilon Q_2-\epsilon^4Q_3+\epsilon^2 Q_4,\\ \notag
Q_3' &=& \epsilon^2Q_1-\epsilon^4Q_2-\epsilon Q_3+\epsilon^3 Q_4,\\ \notag
Q_4' &=& \epsilon Q_1-\epsilon^2Q_2-\epsilon^3Q_3+\epsilon^4 Q_4. \notag
\end{eqnarray}
The symmetric matrix defining the quadratic form
$y_1Q_1'+y_2Q_2'+y_3Q_3'+y_4Q_4'$ is equal to the matrix
$$A(y_1,y_2,y_3,y_4) = \begin{pmatrix}a_{11}&a_{12}&a_{13}&0\\
a_{21}&a_{22}&0&a_{24}\\
a_{31}&0&a_{33}&a_{34}\\
0&a_{42}&a_{43}&a_{44}\end{pmatrix},$$
where
\begin{eqnarray*}
a_{11} &=& -(\epsilon^3y_1+\epsilon y_2+\epsilon^4y_3+\epsilon^2 y_4,\\
a_{12}=a_{21} &= &-(\epsilon^2y_1+\epsilon^4 y_2+\epsilon y_3+\epsilon^3 y_4),\\
a_{13} = a_{31} &=& \epsilon y_1+\epsilon^2y_2+\epsilon^3 y_3+\epsilon^4 y_4,\\
a_{22} &=&3(\epsilon y_1+\epsilon^2y_2+\epsilon^3y_3+\epsilon^4 y_4),\\
a_{24}&=&a_{42}=\epsilon^4y_1+\epsilon^3y_2+\epsilon^2y_3+\epsilon y_4,\\
a_{33}&=&3(\epsilon^4y_1+\epsilon^3y_2+\epsilon^2y_3+\epsilon y_4),\\
a_{34}&=&a_{43}=\epsilon^3y_1+\epsilon y_2+\epsilon^4y_3+\epsilon^2 y_4,\\
a_{44}&=&-(\epsilon^2y_1+\epsilon^4y_2+\epsilon y_3+\epsilon^3 y_4).
\end{eqnarray*}
Plugging in these expressions and computing the determinant, we obtain 
(using MAPLE) the equation of the discriminant surface, where $y_5=-y_1-y_2-y_3-y_4$,
$$\det(A(y_1,y_2,y_3,y_4)) = 
25(\sum_{i=1}^5y_i^4-\sum_{1\le i< j\le 5}^5y_iy_j^3+
\sum_{1\le i< j\le k\le 5}y_iy_jy_k^2-3\sum_{1\le i<j<k<l\le 5}y_iy_jy_ky_l),$$
where $y_5 = -(y_1+y_2+y_3+y_4).$
We can rewrite this in terms of power functions to obtain the equation of the discriminant surface in the form
\beq\label{seven}
30 \sum_{i=1}^5y_i^4-7(\sum_{i=1}^5y_i^2)^2 = y_1+y_2+y_3+y_4+y_5 = 0.
\eeq
This is the equation of the 10-nodal surface $S_{7/30}$ from the Hashimoto pencil.\footnote{This identification of the discriminant surface with the surface $S_{7/30}$ was also confirmed to me by S. Mukai}

\begin{remark} The 10 nodes of the symmetric discriminant quartic (also known as a 
\emph{Cayley symmetroid quartic}) correspond to the singular lines of 10 reducible quadrics in 
the web. According to A. Coble \cite{Coble}, p. 250, they are the ten common secants of the rational normal 
curves $C_1$ and $C_2$. 
\end{remark}

Let us now compute the equation of the Steinerian surface of the web of quadrics defined by the map \eqref{web}. If we choose the basis of $\phi(\bbW^4)$ given by \eqref{basis1}, then the equation of the Steinerian surface is given by 
the determinant of the matrix with columns $Q_1'\cdot \bfx,\  Q_2'\cdot \bfx,\ Q_3'\cdot \bfx,\ Q_4'\cdot \bfx$, where we identify $Q_i'$ with the associated symmetric matrix and $\bfx$ is the column of the coordinates $x_0,x_1,x_2,x_3$ in $\bbU_4$:
{\small $$A = \begin{pmatrix}-\epsilon^3x_0-\epsilon^2x_1+\epsilon x_2&
 \epsilon^4x_3-\epsilon^2x_0+3\epsilon x_1&
  3\epsilon^4x_2+\epsilon^3x_3+\epsilon x_0& \epsilon^4x_1+\epsilon^3x_2-\epsilon^2x_3\\
 -\epsilon^4x_1+\epsilon^2x_2-\epsilon x_0& -\epsilon^4x_0+\epsilon^3x_3+3\epsilon^2x_1& 3\epsilon^3x_2+\epsilon^2x_0+\epsilon x_3& -\epsilon^4x_3+\epsilon^3x_1+\epsilon x_2\\ 
 -\epsilon^4x_0+\epsilon^3x_2-\epsilon x_1& 3\epsilon^3x_1+\epsilon^2x_3-\epsilon x_0& \epsilon^4x_3+\epsilon^3x_0+3\epsilon^2x_2& \epsilon^4x_2+\epsilon^2x_1-\epsilon x_3\\
 \epsilon^4x_3-\epsilon^3x_1-\epsilon^2x_0& 3\epsilon^4x_2-\epsilon^3x_0+\epsilon x_3& \epsilon^4x_1+\epsilon^2x_3+3\epsilon x_2& -\epsilon^3x_3+\epsilon^2x_2+\epsilon x_1\end{pmatrix}^t
$$}
Computing the determinant of the matrix $A$, we find that the equation of the Steinerian surface coincides with the equation of the surface $\sfS_4$ from \eqref{s3s4}.

The Steinerian surface $\sfS_4$ contains 10 lines, the singular lines of ten reducible members of the web corresponding to singular points of the discriminant surface $S_{7/30}$. These are of course the ten common secants of the two rational normal curves $C_1$ and $C_2$. The unique orbit of 20 points from the same Proposition is of course the intersection points of the 10 common secants with the curves $C_1$ and $C_2$.

\begin{remark} Recall that a minimal nonsingular model of a Cayley quartic symmetroid is isomorphic to the K3-cover of an Enriques surface. In its turn, the Enriques surface is isomorphic to a Reye congruence of bidegree $(7,3)$ in the Grassmannian $G(2,4)$ (see \cite{Cossec}). This applies to our surface $\sfS_4$. The embedding  $j:\sfS_4\hookrightarrow |\bbU_4|$ defines an invertible sheaf $\calL_1 = j^*\calO_{|\bbU_4|}(1)$. The birational morphism from $s:\sfS_4\to S_{7/30}\subset |\bbW_4|$ whose image is the quartic symmetroid $S_{7/30}$ that inverts the rational map from the discriminant surface to the Steinerian surface defines another invertible sheaf $\calL_2 = s^*\calO_{|\bbW_4|}(1)$. It is known that the covering involution of $\sfS_4$ preserves the tensor product $\calL_1\otimes \calL_2$. Since the both maps are $\frakA_5$-equivariant, we obtain that the action of $\frakA_5$ on $S_{7/30}$ descends to an action on the Enriques surface.
\end{remark}

The following is a list of open questions:
\begin{itemize}
\item Find the values of the parameters corresponding to the 10-nodal members. Are these surfaces determinantal? What is the transcendental lattice of its minimal nonsingular member.
\item Find  more facts about the $\frakS_5$-invariant surface $\sfS_3$. Is it determinantal? What is its transcendental lattice?
\item Find the transcendental lattices of the general member of the pencil.
\end{itemize}

\section{The catalecticant quartic surface}
There is another view of the Cayley symmetroid quartic  surface $S_{7/30}$. First we have the decomposition of the linear representations of $\frakA_5$
$$S^6\bbV^\vee \cong S^6\bbV \cong \bbW_4\oplus S^2\bbV.$$
Comparing this with \eqref{decomp1}, we find an isomorphism
$$S^2(\bbU_4)^\vee\cong S^6\bbV\oplus S^2\bbV'.$$
The summand $S^6\bbV$ is isomorphic to  the image of $S^2(\bbU_4)^\vee$ under the restriction to the Veronese curve $R_6$ of degree $3$. The kernel $S^2\bbV'$ is isomorphic to the (linear) space of quadrics vanishing on the Veronese curve.
Recall that each $f\in S^{2d}\bbV^\vee$ defines a linear map $a:S^d\bbV\to S^d\bbV^\vee$ (the \emph{apolarity map}). We view a basis of $S^d\bbV$ as partial derivatives of the coordinates $(u,v)$ in $S^d\bbV$ and apply the differential operator $f_d(\frac{\partial}{\partial u},\frac{\partial}{\partial v})$ to $f$ to obtain a binary form of degree $d$ in $u,v$. In these coordinates, the determinant of the map is a polynomial of degree $d+1$ in coefficients of the form $f$, called the 
\emph{catalecticant} (see \cite{CAG},1.4.1). A general zero of this polynomial is a binary form of degree $2d$ that can be written as a sum of less than the expected number (equal to $d$) of powers of linear forms. In the projective space $|S^{2d}\bbV^\vee|$ this corresponds to the  variety of $d$-secant subspaces of dimension $d-1$ of the Veronese curve $R_{2d}$. In our case  where $d = 3$ and the basis in $S^6\bbV^\vee$ is taken as in \eqref{binary}, we get a quartic polynomial
$$\textup{Cat} = \det\begin{pmatrix}a_0&a_1&a_2&a_3\\
a_1&a_2&a_3&a_4\\
a_2&a_3&a_4&a_5\\
a_3&a_4&a_5&a_6\end{pmatrix}.$$
The zeros of this polynomial in $|S^6\bbV^\vee|$ is the  variety $\textrm{Tri}(R_6)$ of trisecant planes of the Veronese curve $R_6$. Let $S = \textrm{Tri}(R_6)]\cap |\bbW_4|$, this is a quartic surface in $|\bbW_4|$. Let us see that it coincides with the Cayley symmetroid $S_{7/30}$ studied in the previous section. In coordinates,  a direct computation shows  
that binary form 
$f= \sum a_s\binom{6}{s}t_0^st_1^{6-s}\in S^6\bbV^\vee$ corresponds to the quadric 
$$Q= a_0y_0^2+a_2y_1^2+a_4y_2^2+a_6y_3^2+2a_1y_0y_1+2a_2y_0y_2+2a_3y_0y_3$$
$$+2a_5y_2y_3+2a_4y_3y_1+2a_3y_1y_2.$$
The condition that $\textup{Cat}(f) = 0$ becomes the condition that $\textup{Discr}(Q) = 0$. This shows that the catalecticant quartic becomes isomorphic to the discriminant quartic in $\bbP^6 =\bbP(N^\perp)$. The singular locus of the variety $\textrm{Tri}(R_6)$ is of degree 10 and it is isomorphic to the secant variety of $R_6$. Intersecting it with $|\bbW_4|$ we obtain 10 singular points of our symmetroid $S_{7/30}$. 

Note that  another model of $S^6\bbV$ is the space of \emph{harmonic cubics} used in \cite{Hitchin} in $|S^2\bbV|$ with respect to the dual of the fundamental conic.

 \section{$\frakA_5$-invariant rational plane sextic}
Let $\calN_1$ and $\calN_2$ be the nets of quadrics with base loci rational normal curves $C_1$ and $C_2$ defined by parametric equations \eqref{par1} and \eqref{par2}. Restricting $|\calN_2|$ to $C_1$ we obtain a map 
$C_1\to |\calN_2^\vee| \cong |S^2\bbV'{}^\vee|$. We identify the plane $|S^2\bbV'{}^\vee|$ with the plane $|S^2V'|$ via the $\frakA_5$-invariant conic. Using the basis of $\calN_2$ formed by the quadrics $V(x_0x_3+9x_1x_2), V(x_0^2+3x_2x_3)$ and $V(x_3^2-3x_0x_1)$, the map $\bbP^1\to C_1\to |S^2\bbV'{}^\vee|$ is given by 
$$(u,v) \mapsto (z_0,z_1,z_2) = (10u^3v^3,u^6+3uv^5,v^6-3u^5v).$$
Recall that the quadratic polynomials $r_1,r_2,r_3$ defining the basis are eigenvectors of $S$ with eigenvalues $(1,\epsilon^3,\epsilon^2)$ and hence proportional to Klein's dual coordinates
$(\sfA_0',\sfA_1',\sfA_2')$. Since they are also transformed as $(r_1,r_2,r_3)\mapsto (-r_1,r_3,r_2)$ under $U$, they are equal to $(c'\sfA_0',c\sfA_1',-c\sfA_2')$ for some constants $c,c'$. Also, we know that the dual conic 
$\sfA_0'{}^2+4\sfA_1'\sfA_2' = 0$ parameterizes the singular quadrics in the net, and hence the equation of $S_2$ shows that we may assume that $c' = \pm 1, c = 2$. We noticed before that in our coordinates $u,v$, we have to choose $c' = -1$. 
$$(z_0,z_1,z_2) = (-\sfA_0',2\sfA_1',-2\sfA_2').$$
The polarity isomorphism $|S^2\bbU|\to |S^2\bbU'{}^\vee|$ defined by the  conic $\sfA_0^2+\sfA_1\sfA_2 = 0$, gives  $(\sfA_0,\sfA_1,\sfA_2) = (-2\sfA_0',\sfA_2',\sfA_1')$. Thus, in the Klein coordinates, the image of $\bbP^1\to C_1\to |S^2\bbU'{}^\vee|\to |S^2\bbU'|$ is equal to the curve $\Gamma_1$ with parametric equation 
$$\Gamma_1:(u,v)\mapsto (\sfA_0,\sfA_1,\sfA_2) = (-5u^3v^3,-v^6+3u^5v,u^6+3uv^5).$$
This agrees with the parametric equation of a $\frakA_5$-invariant rational sextic curve found by R. Winger in \cite{Winger1}, \cite{Winger2}. He uses slightly different coordinates $(x,y,z) = (\sfA_1,\sfA_2,\sfA_0)$.
 
The pencil of $\frakA_5$-invariant sextics is spanned by the triple conic and the union of the six fundamental lines: 
$$\mu (xy+z^2)^3+\lambda z(x^5+y^5+z^5+5x^2y^2z-5xyz^3+z^5).$$
 According to R. Winger \cite{Winger2}, the curve $\Gamma_1$ is a unique rational curve  in this pencil with its 10 nodes forming an orbit of $\frakA_5$.  It corresponds to the parameters $(\lambda:\mu) = (27,5)$.  The equation of 
$\Gamma_1$ becomes
$$27z(x^5+y^5)+5x^3y^3+150x^2y^2z^2-120xyz^4+32z^6 = 0.$$

The action of $\frakA_5$ in the dual plane $|S^2\bbU^\vee|$ has an orbit of 6 points. The corresponding lines in the original plane $|S^2\bbU|$ are the six lines with equations
$$\sfA_0 = 0, \ \sfA_0+\epsilon^\nu \sfA_1+\epsilon^{-\nu}\sfA_2 = 0, \ \nu = 0,\ldots,4.$$
Each of these \emph{fundamental lines} intersects the curve $\Gamma_1$ at 2 points with multiplicity 3. The 12 intersection points are the images of the orbit of $\frakA_5$ acting in $\bbP^1$ that can be taken as the vertices of the icosahedron. They are the zeros of the polynomial
$$\Phi_{12} = uv(v^{10}+11u^5v^5-u^{10}).$$
The ten pairs of branches of the singular points of $\Gamma_1$ correspond to the $\frakA_5$-orbit of 20 points in $\bbP^1$. They are the zeros of the polynomial
$$\Phi_{20} = u^{20}+v^{20}+288(u^{15}v^5-u^{5}v^{15}-494u^{10}v^{10}.$$

  The dual curve $\Gamma_1^*$ of $\Gamma_1$ is a rational curve of degree 10 with parametric equation
$$(x^*,y^*,z^*) = (-10u^7v^3-5u^2v^8,5u^8v^2-10u^3v^7,u^{10}-14u^5v^5-v^{10}).$$
Its equation can be found in  \cite{Hitchin}, p. 83. Not that, via the fundamental conic, we can identify the dual planes $|S^2\bbV|$ and $|S^2\bbV^\vee|$. Thus we have the second  rational sextic curve $\Gamma_2$ and its second dual curve $\Gamma_2^*$. It  is a rational curve $\Gamma_2^*$ of degree 10 with parametric equation
$$(x,y,z) = (-10(u^7v^3-u^2v^8,10(u^8v^2-2u^3v^7,u^{10}-14u^5v^5-v^{10}).$$
The pair of these curves corresponds to the pair of rational normal curves $C_1,C_2$. These pairing of rational normal sextics is discusses in details in \cite{Coble}, Chapter 4.

Since the line $z= 0$ intersects $\Gamma_1$ at two points with multiplicity 3, the Pl\"ucker formulas show that $\Gamma_1^*$ and, hence $\Gamma_2^*$ has six $4$-fold multiple points. They are the fundamental points of $\frakA_5$.\footnote{Note that there is another $\frakA_5$-invariant rational curve of degree 10 with 36 double points, we will not be concerned with it.}

\begin{remark} The blow-up of the 10 nodes of $\Gamma_1$ (or of $\Gamma_2$) is a rational \emph{Coble surface} $C$ with $|-K_C| = \emptyset$ but $|-2K_C| \ne \emptyset$  and consists of the proper transform of the sextic $\Gamma_1$. It inherits the $\frakA_5$-symmetry of $\Gamma_1$. It is known that a Coble surface with an irreducible anti-bicanonical curve is a degenerate member of a pencil of Enriques surfaces. As I was informed by S. Mukai, there is a pencil of Enriques surfaces containing $\frakA_5$ in its automorphism group 
(acting linearly in its Fano embedding in $\bbP^5$) that contains among its members the Enriques surfaces with K3-covers $S_{1/2}$ and $S_{7/30}$ as well as the Coble surface $C$.
\end{remark} 

\begin{remark} The parametric equation of the $\frakA_5$-invariant rational sextic $\Gamma_1$ appears in \cite{Hulek2}, p. 122. It is shown there that the curve $\Gamma_1$ is equal to the intersection of the tangent lines at the origin of the modular family of elliptic quintic curves in $\bbP^4$ with  the eigenplane of the negation involution. It is also discussed in \cite{Barth}, pp. 751-752,  where it is shown that there is a $(3:1)$-map from the modular Bring's curve 
$X_0(2,5)$ (isomorphic to the intersection of the Clebsch diagonal cubic surface with the quadric $x_1^2+\cdots+x_5^2 = 0$ to the normalization of the curve $\Gamma_1$ that coincides with the forgetting map $X_0(2,5)\to X(5)$. The twelve cusps of $\Gamma_1$ are its intersection points with the fundamental conic.
\end{remark}

 \end{document}